\documentclass[12pt]{article}
\usepackage{amsmath,amsthm,amsfonts,fullpage}
\usepackage[sc]{mathpazo}

\newtheorem{theorem}{Theorem}[section]

\newtheorem{lemma}[theorem]{Lemma}
\newtheorem{proposition}[theorem]{Proposition}
\newtheorem{conjecture}{Conjecture}
\newtheorem{question}[conjecture]{Question}
\newtheorem{problem}[conjecture]{Problem}

\theoremstyle{definition}

\newcommand{\tr}{\operatorname{tr}}
\newcommand{\ao}{\text{\rm AO}}
\newcommand{\sao}{\text{\rm SAO}_\beta}
\newcommand{\eps}{\varepsilon}
\newcommand{\trb}{\tfrac{2}{\sqrt\beta}}

\DeclareMathOperator{\Ai}{Ai}
\newcommand{\mat}[4]{\left( \begin{array}{cc}
#1 & #2  \\
#3 & #4  \\
\end{array} \right)}
\newcommand{\pr}{P}
\newcommand{\ev}{E}

\newcommand{\cN}{\mathcal N}
\newcommand{\Sineb}{\operatorname{Sine}_{\beta}}
\newcommand{\sch}{{\sf Sch}_\tau}
\newcommand{\Prob}{P}
\newcommand{\of}[1]{\left ( #1 \right ) }
\newcommand{\eqd}{\stackrel{d}{=}}

\newcommand{\set}[1]{\left\{#1\right\}}

\title{Operator limits of random matrices}

\author{B\'alint Vir\'ag\footnote{Appeared in the Proceedings of the International Congress of Mathematicians, Seoul, 2014}
}

\begin{document}

\bibliographystyle{plain}


\date{May 25, 2014}
\maketitle

\begin{abstract}
We present a brief introduction to the theory of operator limits of random matrices to non-experts. Several open problems and conjectures are given. Connections to statistics,
integrable systems, orthogonal polynomials, and more, are discussed.
\end{abstract}

\section{Introduction}

Wigner introduced random matrices to mathematical physics as a model
for eigenvalues in a disordered system, such as a large nucleus. In the classical approach to random matrices, one considers some statistic of the matrix, and tries to understand the large $n$ limit.

Here we follow a different approach. It is along the lines of the ``objective method'' coined by David Aldous. The goal is to take a limit of the entire object of interest, in this case the matrix itself. This has the advantage that the structure in the matrix will be preserved in the random limit. This method has been very successful in understanding
random objects, notable examples are (the classical) Brownian motion, the continuum random tree, the Brownian map, and SLE, and recent limits of dense and sparse graphs.

This study of random matrices was initiated by the predictions in the work of
Edelman and Sutton \cite{ES07}. They suggested that the tridiagonal matrix models introduced by Trotter \cite{Trotter} and Dumtiriu and Edelman \cite{DE}, should have certain differential operator limits. Their work was the starting point of intense activity in the area, which is what this paper intends to review.

We will first introduce the tridiagonal models. Then we consider various operator limits and discuss some applications.

\section{Tridiagonal models}
\label{s:tridiagonal}
Trotter never thought that his 1984 paper \cite{Trotter}, in which he introduced tridiagonalization to the theory of random matrices, would ever be very important. Indeed, he just used it to give a different proof for the Wigner semicircle law for the GOE, of which there are (and had been) a plethora of other proofs. His proof was nevertheless beautiful, and we will present a quick modern version in Section \ref{s:rooted}.

Tridiagonalization is a method to find eigenvalues of self-adjoint matrices that is still used in modern software, for example in the Lanczos algorithm.  It is also useful if we want to store the eigenvalues of an $n\times n$ matrix, but not $n^2$ data points, without operations beyond linear algebra.

Starting with an $n\times n$ symmetric matrix $A$, first conjugate it with a special block orthogonal matrix so that its first coordinate vector is fixed. Writing both matrices in the block form
$$
\mat{1}{}{}{O}\mat{a}{b^\dagger}{b}{C}\mat{1}{}{}{O^\dagger}=\mat{a}{(Ob)^\dagger}{Ob}{OCO^{\dagger}}
$$
so one can choose $O$ so that $b$ becomes a nonnegative multiple of the first coordinate vector, and the first row is like that of a tridiagonal matrix. One can iterate this procedure (conjugating by an orthogonal matrix fixing the first $k$ coordinates in the $k$th step), to get a tridiagonal matrix.

The Gaussian orthogonal ensemble (GOE) is the random matrix $A=(M+M^t)/\sqrt{2}$ where $M$ has independent standard Gaussian entries. It has the property that conjugation by an orthogonal matrix preserves its distribution.

Exploiting this property and independence, we see that the result of tridiagonalization is a symmetric matrix with independent diagonals $a_i$, (resp. off-diagonals $b_i$). Setting $\beta=1$ and dividing by $\sqrt{n\beta}$ we get the tridiagonal matrix $T$ with entries
\begin{equation}\label{e:bhermite}
a_i\sim N(0,2/n\beta), \qquad b_i\sim \chi_{(n-i)\beta}/\sqrt{n\beta}.
\end{equation}
(Recall that $\chi_k$ is the distribution of the length of an $n$-dimensional vector with independent standard normal entries.
Starting with standard complex normals gives the Gaussian unitary ensemble (GUE) and the same story with $\beta=2$.
It will be convenient to consider the resulting joint
density for the variables $a_i$, $\log b_i$ as a constant times
\begin{equation}\label{chi}
\exp (-\mbox{$\frac{\beta}{4}$} n \tr V(T))
  \times \prod_{k=1}^{n-1}b_k^{\beta (n-k)}
\end{equation}
with $V=x^2$.

The tridiagonalization procedure seem to produce a non-unique result (there are many choices for the orthogonal matrices), but this is not the case. If the vectors $e, Ae, \ldots A^{n-1}e$ are linearly independent, we always get the {\it same } Jacobi matrix (tridiagonal with positive off-diagonals). It is, in fact the matrix $A$ written in the Gram-Schmidt orthonormalization of this basis.

In both descriptions, $T$ is an orthogonal conjugate to $A$, with the first coordinate vector fixed. If one defines this as an equivalence relation on symmetric matrices where $e$ is cyclic, then each class contains exactly one Jacobi matrix, so they are natural class representatives.

So $T$, with $2n-1$ parameters, encodes the $n$ eigenvalues of $A$. But what else does this encode? Check that $$A^k_{11}=T^k_{11}=\int x^k d\sigma,$$ for the {\bf spectral measure }
$$\sigma =\sum_{i=1}^n q_i \delta_{\lambda_i},$$ where $q_i=\varphi_{i,1}^2$ for the normalized eigenvectors $\varphi_i$.
So $T$ encodes the spectral measure, which is a probability measure supported on $n$ points and so are described by $2n-1$ parameters.

Since for the GOE the eigenvectors are uniform on the unit $n$-sphere and independent of the eigenvalues, we can write the joint density on $\lambda_i,\log q_i$ as a constant times
\begin{equation}\label{dbeta}
\exp (-\beta n \tr V(A))
 \times \prod_{i<j} |\lambda_i-\lambda_j|^\beta \prod_{k=1}^{n}q_i^{\beta/2}
\end{equation}
using the well-known formula for the eigenvalue distribution \cite{AGZ}.
Now the factors the left of $\times$ in \eqref{chi} and \eqref{dbeta} are equal, since $A,T$ have the same eigenvalues. Interestingly, the same holds for the value on the right, see Section 3.1 of \cite{Deift}! Since it is also known that the map
\begin{equation}\label{e:chv}
(a_1, \ldots, a_n,\log b_1, \ldots \log b_{n-1}) \mapsto (\lambda_1,\ldots, \lambda_n, \log q_1,\ldots \log q_n)
\end{equation}
is a bijection, it follows that it
is {\bf measure-preserving} (up to a fixed constant). As a consequence, the equivalence of measures \eqref{chi}, \eqref{dbeta} holds for all functions $V$ and $\beta>0$. When $V=x^2$, the model is called the {$\beta$-Hermite ensemble} and this was shown with the same methods by Dumitriu and Edelman \cite{DE}. Just as in the special cases of the GOE and GUE, the tridiagonal matrix $T$ has independent entries.

This model \eqref{dbeta} on $n$ points is called Dyson's beta ensemble.

\bigskip \noindent {\bf Structure of the tridiagonal matrices.}
As one expects, various features of the eigenvalue distribution can
be read off the tridiagonal matrix $T$. For example, the top (and bottom)
eigenvectors of the matrix have all of their $\ell^2$ mass in the
first order $n^{1/3}$ coordinates. So in order to understand edge statistics,
one can take a scaling limit of this part of $T$.

Similarly, for the $\beta$-Hermite $T$ eigenvectors for eigenvalues near $0$
have their $\ell^2$-mass distributed through the whole length $n$. So bulk local statistics of eigenvalues will be understood
by taking an operator limit of $T$ on this scale.

So while local eigenvalue statistics have to do with the
global structure of $T$, the global statistics of eigenvalues (such as
the Wigner semicircle law) have to do with the local structure of $T$ at a random
vertex, as we will see next. The spectral measure at the first coordinate
is also closely related to the eigenvalue distribution.

\section{Density of states}\label{s:rooted}

In this section, we pursue the point of view of operator limits to deduce the Wigner semicircle law. In fact, we  will get two proofs, one using rooted convergence of graphs, and the other using Benjamini-Schramm convergence.

\bigskip\noindent{\bf Rooted convergence and the Wigner semicircle law.} A sequence of edge-labeled, bounded degree rooted graphs $(G_n,o)$ is said to converge locally to a rooted graph $G$ if for every $r$, the $r$ neighborhood of $o$ the graph stabilizes and the labels in the neighborhood converge pointwise as $n\to\infty$.

For example, using the asymptotics $$\chi_n \approx \sqrt{n} + N(0,1/2),$$ we see that the $\beta$-Hermite ensemble matrix $T=T_n$ (thought of as weighted adjacency matrix) rooted at the first vertex
converges almost surely locally to the graph $T^*$ of the nonnegative integers (with weights 1) as $n\to\infty$ and $\beta$ is fixed.

Here we identity the graphs with their adjacency matrices. Recall the spectral measure of $G$ at $o$ is the measure whose $k$-th moments are $G^k_{o,o}$. The method of moments shows that rooted convergence implies convergence of spectral measures at the root $o$.

The moments of the spectral measure of $T^*$ at $o$ these are the number of returning simple random walk paths that stay nonnegative; they characterize the Wigner semicircle law.

What we have shown is that the spectral measures converge almost surely. But the spectral measure assigns Dirichlet$(\beta/2,\ldots \beta/2)$ weights to the eigenvalues, see \eqref{dbeta}. The law of large numbers for these weights shows that the empirical eigenvalue distribution has the same limit.

An argument like this works for more general potentials $V$ -- in this case the limiting rooted labeled graph is the Jacobi operator associated to the orthogonal polynomials with respect to the measure $e^{-V(x)}\,dx$, see \cite{KRV}.

\bigskip
\noindent{\bf Benjamini-Schramm limits and the Wigner semicircle law.} Here we deduce the semicircle law in a way which is, essentially, equivalent to Trotter's \cite{Trotter} but uses no computation. A sequence of unrooted, labeled finite graphs $G_n$ is said to to a random rooted graph $(G,o)$ in the Benjamini-Schramm sense if the law of $(G_n,o)$ converges there with uniform choice of $o$. The convergence is with respect to the topology of rooted convergence introduced above.

Again, the method of moments shows that the expected spectral measure at $o$, which is the empirical eigenvalue distribution of $G_n$, converges to the expected spectral measure of $(G,o)$ at $o$.

A moment of thought shows that the almost sure Benjamini-Schramm limit of the $\beta$-Hermit ensembles is
$
\sqrt{U} \mathbb Z
$, where $\mathbb Z$ is the graph of the integers, rooted at $o$, $U$ is a uniform random variable that comes from the mean of the $\chi$ variable at the uniformly chosen location of the root.

Now $\mathbb Z$ is also the Benjamini-Schramm limit of  $n$-cycles, whose eigenvalues are the real parts of equally spaced points on the circle $\{|z|=2\}\subset \mathbb C$. Hence the spectral measure of $\mathbb Z$ is the real part of uniformly chosen point on the circle of radius $2$.

The expected spectral measure $\mu$ of $\sqrt{U}\mathbb Z$ is thus the real part of the uniformly chosen point from a random circle with radius $2\sqrt{U}$; but this is just another way to chose a point from uniform measure in the disk of radius two. Thus $\mu$ is the semicircle law.

\section{The $\beta$-Hermite random measure on $\mathbb R$}
\label{s:randommeasure}

A special property of the $\beta$-Hermite matrices $\sqrt{n}\,T_n$ is that they
are minors of each other; as a result, they are the minor of
a semi-infinite Jacobi matrix $J=J_\beta$.

The $\beta\to\infty$ limit $J_\infty$ has zeros
on the diagonal and $\sqrt{k}$ at positions $(k
+1,k)$ and $(k,k+1)$. Its spectral measure at the first coordinate is standard normal.

Such matrices have relevance in the theory of orthogonal polynomials.
Here we review a few brief facts. Given a measure $\mu$ with infinite support on $\mathbb R$ with sufficiently
thin tails, the $k$th orthonormal polynomial is the unique degree $n$ polynomial
with positive main coefficient that is orthogonal in $L^2(\mathbb R,\mu)$ to all lower degree polynomials.

One can show that there are unique $a_n$ and $b_n>0$ so that the $p_n$ satisfy
a recursion $p_{k-1}b_{k-1} + p_{k}a_{k}+p_{k+1}b_{k} = x p_{k}$. In
other words, the (not necessarily $\ell^2$) vector $p(x)=(p_k(x))_{k\ge 0}$ satisfies
the eigenvector equation $Jp(x)=xp(x)$ where $J$ is the infinite tridiagonal matrix
built from the $a$-s and $b$-s. Note that here it is crucial that the numbering
is reversed compared to \eqref{e:bhermite}.

Note that $p(x)$ restricted to the first $n$ coordinates is an eigenvector of the
$n\times n$ minor of $J$ if and only if $p_n(x)=0$. In particular, the
$p_n$ are constant multiples of the characteristic polynomials of this minor.

Conversely, given such $J$ and assuming that it is self-adjoint,
one can recover the measure $\mu$ as the spectral measure of $J$
at the first coordinate. Since $J_\beta$ is easily
shown to be self-adjoint, we have shown

\begin{theorem}[Coupling of the $\beta$-Hermite ensembles]
There exists a random measure $\mu_\beta$ so that for all $n$ the zeros of the orthogonal polynomial $p_n$
with respect to $\mu_\beta$ are distributed as the eigenvalues of
the $n$-point $\beta$-Hermite ensemble.
\end{theorem}
It also follows that the $\beta$-Hermite eigenvalues are exactly the Gaussian quadrature points for this measure!

The measure $\mu_\beta$ can be thought of as a random ``rough'' version of the standard normal
distribution ($\mu_\infty$). The measure has been studied by Breuer, Forrester, and Smilansky \cite{BFS}. They showed that its Hausdorff dimension
is almost surely equal to $(1-2/\beta)^+$. For $\beta<2$, the measure is pure point. A similar phenomenon holds for the family of Gaussian multiplicative cascade
measures, see, for example \cite{rhodes2013gaussian} in some sense it is a noncommutative version.  A natural question
is the following
\begin{question}[Spectral measure and multiplicative cascades]
Does the $\beta-Hermite$ measure and the Gaussian multiplicative
cascade measure with the same Hausdorff dimension have the same fractal spectrum?
\end{question}

\begin{question}[Nested models]
Can any other Dyson $\beta$-ensembles be coupled this way?
How about other natural random matrix models?
\end{question}

\section{Edge limits and the stochastic Airy operator}

For $n$ large and $k=o(n)$, we have the asymptotics
$\chi_{n-k} \asymp \sqrt{n}-k/\sqrt{4n}+ N(0,1/2)$. Thus the top minor of size $o(n)$
of $(2I-T)$ looks like a discrete second derivative plus multiplication by
$2k/n$, plus multiplication by discrete independent noise. The precise continuous analogue would be
\begin{equation}\label{sao}
\sao = -\partial_t^2 + t + \trb b'
\end{equation}
called the Stochastic Airy Operator, where $b'$ is a distribution (the derivative of standard Brownian motion). Edelman and Sutton \cite{ES07} conjectured that this operator, acting on $L^2(\mathbb R^+)$ with Dirichlet boundary conditions $f(0)=0$, is the edge limit of $T_n$. This was proved in in \cite{RRV}:
\begin{theorem}\label{t:rrv}
There exists a coupling of the $\beta$-Hermite random matrices $T_n$ on the same probability space so that
a.s. we have
$$ n^{2/3}(2  I -T_n)\to \sao $$
in the norm-resolvent sense: for every $k$ the bottom $k$th eigenvalue converges the and corresponding eigenvector converges in norm.
Here $2 I -T_n$ acts on the embedding $\mathbb R^n\subset L^2(\mathbb R_+)$ with coordinate vectors $e_j=n^{1/6}{\mathbf 1}_{[j-1,j]n^{-1/3}}$.
\end{theorem}

The limiting distribution of the top eigenvalue of the GOE, and GUE are called the Tracy-Widom distribution TW$_\beta$ with $\beta=1,2$, respectively. It follows that for $\beta=1,2$ the negative of the bottom eigenvalue $-\Lambda_0$ of $\sao$ has TW$_\beta$ distribution. For more general $\beta$, this can be taken as a definition of TW$\beta$.

The domain of $\sao$ can be defined precisely (see \cite{bloemendal2011finite}), but we will not do that here. The eigenvalues and eigenvectors can be defined though the Courant-Fisher characterization,
$$
\Lambda_k = \inf_{A:\dim A =k+1} \sup_{f\in A, \|f\|_2=1} \langle f,\sao f\rangle.
$$
the latter can be defined via integration by parts as long as $f$, $f'$ and $\sqrt{t}f$ are in $L^2(\mathbb R^+)$, and in the formula $A$ is a subspace of such functions. The eigenvectors are defined as the corresponding minimizers, and can be shown to be unique, see \cite{RRV}.

\begin{proof}[Glimpses of the proof of Theorem \ref{t:rrv}.]
We explain how to show that the bottom eigenvalue converges (see \cite{RRV} for the rest).
It is a nice exercise \cite{RRV} to show that given a Brownian path and $\eps>0$
there is a random constant $C$ so that for every function
$f$ with $f,f',\sqrt{t}f  \in L^2(\mathbb R)$ we have
$$
|\int f^2\,dB |\le C \|f\|^2 + \eps( \|f'\|^2 + \|\sqrt{t}f\|^2) =
C\|f\|^2+
\eps \langle f, \ao f \rangle.
$$
where $\ao=\operatorname{SAO}_\infty$ is the usual Airy operator $-\partial_t^2 + t$.
In other words, we have the positive definite order of operators
\begin{equation}\label{mainbound}
-C + (1-\eps) \ao \le \sao \le (1+\eps) \ao+ C
\end{equation}
Using Skorokhod's representation and the central limit theorem, we can guarantee a coupling so that the integrated potential of $2I-T_n$ converges uniformly on compacts to that of $\sao$. Moreover, the discrete analogues of the bound \eqref{mainbound} will hold with uniform constants $C$ and all $n$.
Note that taking the bottom eigenvector $f_0$ of $\sao$ and plugging it into the approximating operators, the Rayleigh quotient formula shows that their bottom eigenvalues satisfy
$$\limsup \lambda(n) \le \Lambda_0$$

Conversely, $\sao$ can be tested against any weak limit of the bottom eigenfunctions $f(n)$, which must exist because of the discrete version of \eqref{mainbound} guarantees enough tightness. As a result,
\[\liminf \lambda(n) \ge \Lambda_0.\qedhere \]
\end{proof}
A different operator appears at the so-called hard edge, see \cite{hardedge}, and \cite{ramirez2011hard} for further analysis.

\section{Applications of the stochastic Airy operator}

The stochastic Airy operator is a Schr\"odinger-type operator,
and therefore tools from the classical theory are applicable.

First, as a self-adjoint operator, one can use Rayleigh quotients or positive definite
ordering to characterize its low-lying eigenvalues. Second, as a Schr\"odinger
operator, one can use oscillation theory for the same. We will briefly
show how these methods work.

\begin{theorem}
Let $\Lambda_k\uparrow$ be the eigenvalues of $\sao$. Then almost surely
$$
\lim_{k\to\infty} \frac{\Lambda_k}{k^{2/3}} = \left( \frac{3\pi}{2}\right)^{2/3}
$$
\end{theorem}
\begin{proof}
As a consequence of \eqref{mainbound}, that  inequality \eqref{mainbound} also holds when we replace the operators $\ao$, $\sao$ by
their $k+1$st eigenvalues $\mathcal A_k,\;\Lambda_k$. By letting $\eps\to 0$
we see that $\Lambda_k/\mathcal A_k \to 1$ a.s.  Now note that eigenfunctions
of $\ao$ are translates of the solution $\Ai$ of the Airy differential equation
\begin{equation}\label{Ai}
(-\partial_t^2 +t)\Ai = 0 ,\qquad \Ai(t)\to 0 \mbox{ as }t\to\infty
\end{equation}
by some
$a$ so that $\Ai(-a)=0$. The classical asymptotics of the zeros of $\Ai$
now imply the claim.
\end{proof}

\noindent{\bf Applications of the Rayleigh quotient formula.} Next, we show an argument from \cite{RRV} that gives a sharp bound on the sub-Gaussian left tail of the $TW_\beta$ distribution of $-\Lambda_0$. It only relies on Rayleigh quotients and standard Gaussian tail bounds!

\begin{lemma}\label{l:upperleft}
\begin{eqnarray*}
  P ( \Lambda_0 > a )\le \exp \Big( - \frac{\beta}{24} \,a^3 (1+o(1)) \Big).
\end{eqnarray*}
\end{lemma}
\begin{proof} The Raleigh quotient
formula implies that
$$
\Lambda_0 > a \qquad \Rightarrow \qquad \langle f,\sao f \rangle >  a
$$
for all nice functions $f$. Note that any fixed $f$ will give a bound, and $\langle f,\sao f \rangle$ is just a Gaussian random variable with mean $\|f'\|_2^2+\|f\sqrt{t}\|_2^2$ and variance
$\frac{4}{\beta}\|f\|^4_4$.  In the quest for a good $f$ one expects the optimal $f$ to be relatively ``flat" and ignore the $\|f'\|^2_2$ term. In the tradition of zero-knowledge proofs, it is legal to hide the resulting variational problem and how to solve it from the reader (see \cite{RRV} Section 4). Out of the hat comes
$$
f(x) =   (x\sqrt{a} ) \wedge  \sqrt{(a - x)^+}\wedge
(a-x)^+,
$$
where the middle term is dominant, while the others control
 $\|f'\|_2$. Then
$$
  a\|f\|_2^2 \sim  \frac{a^3}{2}, \qquad
  \| f'\|_2^2 = O(a), \qquad
  \| \sqrt{x} f \|_2^2  \sim  \frac{a^3}{6} , \qquad
   \| f \|_4^4 \sim  \frac{a^3}{3}  .
$$
The proof is completed by substitution, with a standard normal $N$,
\begin{equation*}
  P ( \Lambda_0 > a ) \le P  \left( \frac{2}{\sqrt{3\beta}} \, {a}^{3/2} \, N>
   a^3\left(\frac{1}{2}-\frac{1}{6}+o(1)\right) \right)
  = \exp \Big( - \frac{\beta}{24} \,a^3 (1+o(1)) \Big).
\end{equation*}
\end{proof}

\noindent{\bf Applications of Sturm-Liouville oscillation theory.} Taking the logarithmic derivative  $W=f'/f$ (also called Riccati transformation) transforms the eigenvalue equation $\sao f= \lambda f$ to a first order non-linear ODE. We write this in the SDE form
\begin{equation}\label{SDE}
dW = \trb\, db + \left(t-\lambda-W^2\right)dt
\end{equation}
this can be thought of as an equation on the circle compactification of $\mathbb R$: a solution that explodes to $-\infty$ in finite time should continue from $+\infty$.
In this sense,
the solution is monotone in $\lambda$: increasing $\lambda$ moves it the ``down'' direction on the circle.

Let's first restrict the operator to a finite interval $[0,\tau]$ with Dirichlet boundary condition. Then $\lambda$ is an eigenvalue iff an explosion happens at $\tau$, and increasing $\lambda$ moves the explosions to the left. On $(0,\tau)$ we thus have
\begin{equation}
\#\{\mbox{ explosions }\}  \,= \,\#\left\{ \mbox{ eigenvalues }<\; \lambda\;\right\}.
\end{equation}
For the $\sao$ this statement remains true with $\tau=\infty$, and as a consequence
$$
P(W_\lambda\mbox{ never explodes})=P(\lambda<\Lambda_0).
$$
Let $P_{t,w}$ denote the law of the solution $W$ of the $\lambda=0$ version of \eqref{SDE} started at time $t$ and  location $w$.  Setting
$$
F(t,w)=P_{t,w}(W\mbox{ never explodes }),
$$
we see that the translation invariance of \eqref{SDE} implies that
$$
\lim_{w\uparrow \infty} F(-\lambda,w) =P(\lambda<\Lambda_0).
$$
This gives a characterization for the Tracy-Widom distribution
TW$_\beta$ of $-\Lambda_0$. Boundary hitting probabilities
of
an SDE can always be expressed as solutions of a PDE boundary value problem. Indeed, such functions are martingales and are killed by the generator, see \cite{BV}. So $F$ satisfies
\begin{equation}\label{PDE}
\partial_t F+ \tfrac{2}{\beta}\,\partial_w^2 F + (t-w^2)\,\partial_w F\,= \,0\qquad\text{ for }t,w\in{\mathbb R},
\end{equation}
with
$F(t,w)\to 1$  as $t,w\to\infty$ together, and $F(t,w)\to 0$ as $w\to-\infty$ with $t$ bounded above.

It is easy to check that the problem has a unique bounded solution, and so it gives a characterization of the Tracy-Widom-$\beta$ distribution. However, new ideas were needed to connect these equation to the Painlev\'e systems; before we turn to these, we consider an application of \eqref{SDE} from \cite{RRV}.

\bigskip\noindent{\bf SDE representation and tail bounds.} We now show how the SDE representation \eqref{SDE} is used to attain tail bounds for the law TW$_\beta=-\Lambda_0$ in \cite{RRV}. We prove the matching lower bound to Lemma
\ref{l:upperleft}; readers not familiar with Cameron-Martin-Girsanov transformations may skip this proof.
\begin{lemma}
\begin{eqnarray*}
  P ( \Lambda_0 > a )\ge \exp \Big( - \frac{\beta}{24} \,a^3 (1+o(1)) \Big).
\end{eqnarray*}
\end{lemma}
\begin{proof}
\noindent
By monotonicity of the solutions, we have
\begin{eqnarray*}
\lefteqn{
  P_{\infty,-a} \Bigl( W \mbox{ never explodes} \Bigr)  \ge   P_{1,-a}
  \Bigl(  W \mbox{ never explodes}  \Bigr) } \\
  &  &  \ge  P_{0,-a} \Bigl( W_t \in [0, 2] \mbox{ for all } t\in[-a,0] \Bigr)
 P_{0,0} \Bigl(  W \mbox{ never explodes} \Bigr).
\end{eqnarray*}
The last factor in line two is some positive number not
depending on $a$. To bound the first factor from
below, we first write it using Cameron-Martin-Girsanov formula as
\begin{eqnarray*}
   E_{1,-a} \Bigl[\exp\left( - \frac{\beta}{4} \int_{-a}^0 (t - b^2_t) db_t - \frac{ \beta}{8} \int_{-a}^0 (t - b^2_t)^2 dt \right);  \, b_t \in [0,2] \mbox{ for all } t\le
   0\Bigr],
\end{eqnarray*}
where, for this proof only,  $b_t$ denotes a Brownian motion with diffusion coefficient
$2/\sqrt{\beta}$.  On the event above, the main contribution comes from
$$
   \frac{\beta}{8} \int_{-a}^{0} (t - b^2_t)^2 \,dt  = \frac{ \beta}{24} a^3  + O(a^2)   ,
$$
of lower order is the second term
$$
    \int_{-a}^0 (t  - b^2_t) db_t =  a b_{-a} + \frac{1}{3} (b^3_{-a} - b^3_0) + (\frac{4}{\beta} - 1)  \int_{-a}^0 b_t dt= O(a).
$$
We are left to compute the probability of the event
$$P_{-a,0} ( b_t \in [0,2] \mbox{ for
} t \le 0 ) \ge e^{-ca},$$
since it is the chance of a Markov chain staying in a bounded set for time proportional to $a$. This does not interfere with the main term.
\end{proof}
In \cite{dumaz2011right} arguments of this kind are used to provide a more precise bound for the other tail $P(\Lambda_0<-a)$, including $-3/4$ the exponent in the polynomial correction.  It was shown that
$$P \left(TW_{\beta} > a \right) = a^{-\frac34
\beta+o(1)}\exp\left(-\frac {2} {3} \beta a^{3/2} \right).
$$
See \cite{borot2012right} for further non-rigorous results in this direction.

\bigskip\noindent{\bf Tail estimates for finite $n$.}
It is possible to make versions the tail estimate proofs for finite $n$, before taking the limit.
This was carried out by Ledoux and Rider \cite{ledoux2010small}. They give strong tail estimates for the $\beta$-Hermite (and also Laguerre) ensembles for finite $n$. We quote the $\beta$-Hermite results
from that paper.

\begin{theorem} There are absolute constants $c,C$ so that for all $n\ge 1$, $\eps\in(0,1]$ and $\beta \ge 1$ the  $\beta$-Hermite ensemble $T_n$ satisfies
$$
c^{\beta}e^{-\beta n \eps^{3/2}/c}\; \le \;P\Big(\lambda_1(T_n) \ge 2(1+\eps)\Big) \;\le\; Ce^{-\beta n \eps^{3/2}/C}
$$
and
$$
c^{\beta} e^{-\beta n^2 \eps^{3}/c}\; \le\; P\Big(\lambda_1(T_n) \le 2(1+\eps)\Big) \;\le\; C^\beta e^{-\beta n^2 \eps^{3}/C}
$$
For the second lower bound we need to assume in addition that $\eps < c$.
\end{theorem}

\section{Finite rank perturbations and Painlev\'e systems}

Johnstone \cite{J2} asked how the top eigenvalue changes in a sample covariance matrix if the population covariance matrix is not the identity, but has one (or a few) unusually large eigenvalues?

Similarly, what happens to the Tracy-Widom distribution
when the mean of the entries of the GOE matrix changes? These questions have been extensively studied. In short, perturbations below a critical window do not make a difference, and above create a single unusually large eigenvalue.

For the $\beta=2$ case, \cite{BBP} derived formulas for the deformed Tracy-Widom distributions using Harish-Chandra integrals.
The quest to understand the critical case for $\beta=1$ lead to a simple derivation of the Painlev\'e equations for $\beta=2,4$ in \cite{BV}.

Note that changing the mean of the GOE is just adding a rank-1 matrix. The
GOE is rotationally invariant, so for eigenvalue distributions we may as well add a rank-1 perturbation of the form $e^te$, with the first coordinate vector $e$. Such a perturbation commutes with tridiagonalization.
At criticality, it becomes a left boundary condition for the stochastic Airy operator. The relevant
theorem form Bloemendal and V. \cite{BV} is
\begin{theorem} Let $\mu_n\in{\mathbb R}$. Let $G = G_n$ be a $(\mu_n/\sqrt{n})$-shifted mean $n\times n$ GOE matrix. Suppose that
\begin{equation}\label{Gspike}
n^{1/3}\left(1-\mu_n\right)\,\to\,w\in(-\infty,\infty]\qquad\text{as }n\to\infty.
\end{equation}
Let $\lambda_1>\dots>\lambda_n$ be the eigenvalues of $G$. Then, jointly for $k=0,1,\ldots$ in the sense of finite-dimensional distributions, we have
\[
n^{1/6}\left(\lambda_k - 2\sqrt{n}\right)\,\Rightarrow\, -\Lambda_{k-1}\qquad\text{as }n\to\infty
\]
where $\Lambda_0<\Lambda_1<\cdots$ are the eigenvalues of $\operatorname{SAO}_{\beta,w}$.
\end{theorem}
\noindent Here $\operatorname{SAO}_{\beta,w}$ is the Stochastic Airy operator \eqref{sao} with left boundary condition $f'(0)/f(0)=w$. Similar theorems hold for the other $\beta$-Hermite ensembles perturbed at $e$.

This theorem is useful in two ways. First, it gives a characterization of the perturbed TW laws in terms of a PDE. Conversely, it gives an interpretation of the solutions of a PDE in terms of the perturbed TW laws, giving a fast way to Painlev\'e expressions.

\bigskip\noindent{\bf Painlev\'e formulas.} Let $u(t)$ be the Hastings-McLeod solution of the homogeneous Painlev\'e II equation, i.e.
\begin{equation}\label{PII}
u'' = 2u^3 + tu,
\end{equation}
characterized by
\begin{equation}\label{HM}
u(t)\sim\Ai(t)\quad\text{as }t\to +\infty
\end{equation}
where $\Ai(t)$ is the Airy function \eqref{Ai}. Let
\vspace{-6pt}
\begin{equation}\label{v}{\textstyle
v(t) = \int_t^\infty u^2,\qquad
E(t) = \exp\bigl(-\int_t^\infty u\bigr),\qquad F(t) = \exp\bigl(-\int_t^\infty v\bigr).}
\end{equation}
Next define two functions $f(t,w)$, $g(t,w)$ on ${\mathbb R}^2$, analytic in $w$ for each fixed $t$, by the first order linear ODEs
\begin{equation}\label{w_lax}
\frac{\partial}{\partial w}\begin{pmatrix}f\\g\end{pmatrix}=\begin{pmatrix}u^2&-wu-u'\\-wu+u'&w^2-t-u^2\end{pmatrix}\begin{pmatrix}f\\g\end{pmatrix}
\end{equation}
and the initial conditions
\begin{equation}\label{IC}
f(t,0) \,=\, E(t) \,=\, g(t,0).
\end{equation}

Equation \eqref{w_lax} is one member of the Lax pair for the Painlev\'e II equation. The other pair gives an ODE in the variable $t$. This is now sufficient information to check that $F(t,w)=f(t,w)F(t)$ satisfies the PDE \eqref{PDE}, giving a proof for the Painlev\'e formula $P(\mbox{TW}_2<t) = F(t)$. However, in order to be able to check, we needed to understand where to start looking, and rank-1 perturbation theory helped!

Similar formulas hold for $\beta=4$. For $\beta=1$, Mo \cite{Mo} has developed formulas but we do not know how to check that they satisfy the PDE.
\begin{problem}[Mo's formulas] Find a way to check that Mo's formulas satisfy \eqref{PDE}.
\end{problem}

In \cite{rumanov2012hard} Rumanov finds a new (!) Painlev\'e representation for the
hard edge using the corresponding stochastic operator. But we don't know the bulk analogue, see Question \ref{q:bulkpain}.

\section{Beta edge universality}

The transformation $(\lambda,q)\mapsto (a,b)$ in \eqref{e:chv} turns complicated
dependence into independence in the $\beta$-Hermite case. For more general potentials $V$, the first factor in \eqref{dbeta} is not a product of factors depending on single variables any more, and so the variables are not independent. Still, for quartic $V$ it can be written as a product, where each factor is a function of only two consecutive pairs $(a_i,b_i)$.

This implies that the process  $i\mapsto (a_i,b_i)$ is a Markov chain. Moreover, for general (even) polynomial $V$ it is a $\eta$- Markov with $\eta=\deg V/2-1$, which means that given  $\eta$ consecutive pairs $(a_i,b_i)$ the variables before and after are conditionally independent.

This observation leads naturally to a proof of universality \cite{KRV}. There, it is shown that for $V$ with $V''>c>0$ we have
\begin{theorem}\label{t:main}
There exists a coupling of the random matrices $T=T_n$ on the same probability space and constants $\gamma, \vartheta, \mathcal E$ depending on $V$ only so that
a.s. we have
$$\gamma n^{2/3}(\mathcal E  I -T_n)\to \sao$$
in the norm-resolvent sense. Here $\mathcal E I -T_n$ acts on $\mathbb R^n\subset L^2(\mathbb R_+)$ with coordinate vectors $e_j=(\vartheta n)^{1/6}{\mathbf 1}_{[j-1,j](\vartheta n)^{-1/3}}$. \end{theorem}
\begin{proof}[Proof outline]
In \cite{RRV}, sufficient conditions were given for the convergence of
discrete operators to continuum ones, in particular to  $\sao$. This was done through a more general version of the proof of Theorem \ref{t:rrv}.

The most important condition is that if $\mathcal E$ is the top edge of the equilibrium measure associated with the potential $V$, then the discrete version of the integrated potential converges to the continuum one, locally uniformly:
$$
   n^{1/3}
  \sum_{k= 1}^{ \lfloor t n^{1/3} \rfloor } ( a_k + 2 b_k -\mathcal E) \rightarrow \frac{1}{2} t^2 + \trb b_t
$$
This amounts to having to show a central limit theorem for the $\eta$-Markov chain $(a_i,b_i)$ (we will drop the prefix $\eta$). \begin{itemize}
\item The Markov chain is time-inhomogeneous because of the coefficients of the $b$-terms. However, these change on the scale of order $n$, while

\item the Markov chain mixes exponentially fast, so in logarithmic number of steps it gets to its (local) stationary measure, which can be approximated using a homogeneous version of the problem.

\item the local equilibrium measure is extremely close to Gaussian. Indeed, the joint distribution of stretches of length $n^{1/2-\eps}$ are close in total variation to their Gaussian approximation! So the CLT is true in a very strong sense, and is proved by comparing joint densities.

\item The Markov chain is {\it  not} started from its local stationary distribution at $i=1$. In fact, the first coordinates of the matrix $T$ encode the local equilibrium measure for $V$ just as they do in the $\beta$-Hermite case. Indeed, the limit of the right end of $T$ is the Jacobi operator for the equilibrium measure associated to the potential $V$! See Section \ref{s:rooted}.

\item Thus the CLT as required by the \cite{RRV} criteria does not hold verbatim.  It does hold for $T$ truncated after the first $c\log n$ coordinates, and it can be shown that the truncation does not make a significant difference.
\end{itemize}
\end{proof}

BY now, universality of the $\beta$-ensemble edge eigenvalues has other proofs, some more general,  see \cite{bourgade2013edge}, \cite{bekerman2013transport}. For the Jacobi ensembles, see \cite{holcomb2012edge}.

\begin{question}[Formulas] There exists asymptotic formulas for
correlations and other statistics of the edge and bulk processes,
see for example \cite{desrosiers}. Can these be connected to the limiting operators directly?
\end{question}

\noindent{\bf Exotic edge operators.} We saw in Section \ref{s:rooted}, that empirical distribution of eigenvalues of $T_n$, without scaling, converge to the classical equilibrium measure form potential theory corresponding to $V$.

The convexity and analyticity of $V$ forces this measure to have a density which is decays like $x^{1/2}$ at the edges. As one might guess, this $x^{1/2}$ is crucial for the $\sao$ limit.

When $V$ is analytic, the possible decay rates are $x^{2k+1/2}$ for some integer $k$. The more detailed analysis of universality in \cite{KRV} lead us to the following conjecture. See \cite{RRV} for a more precise version, and a detailed explanation from where the conjectured limit comes from.

\begin{conjecture}
After scaling, $T_n$ converges to the random operator
$$
   \mathcal{S}_{\beta, k} =  - \partial_t^2  + t^{\frac{1}{2k+1}} +  \tfrac{2}{\sqrt{\beta}} \,t^{- \frac{k}{2k+1}} \,b'_t.
$$
\end{conjecture}
For $\beta=2$ the eigenvalue limits have been studied in \cite{Claeys} via the Riemann-Hilbert approach.

\section{Bulk limits -- the Brownian carousel}
\label{s:carousel}

The goal of this section is to describe the limit of the
$\beta$-Hermite ensembles in the bulk.

First, for motivation, we review some history. The nonlinear transformation $(a,b) \to (\lambda,q)$ of Section \ref{s:tridiagonal} is fundamental in several areas, including orthogonal polynomial theory, the Toda lattice, and more generally,
integrable systems and inverse spectral theory. It goes beyond tridiagonal matrices and point measures. A beautiful generalization,
is the theory of {\bf canonical systems}, where the correspondence is between certain matrix-valued ``potentials'' and measures on $\mathbb R$. Canonical systems are a one-parameter families of differential equations of the form
$$
\lambda R_t f=Kf',\qquad \mbox{ on }[0,\eta), \qquad K=\mat{0}{-1}{1}{0}.
$$
where $R$ is a nonnegative definite $2\times 2$ matrix-valued function from $[0,\eta)$, and $f$ takes values in $\mathbb R^2$ on the same interval. When $R$ is invertible everywhere, then the canonical system corresponds to the eigenvalue problem of the Dirac operator
\begin{equation}
  \label{dirac}
R^{-1}K\partial_t
\end{equation}
which is symmetric with respect to the inner product
$$
\langle f,g\rangle= \int_0^\eta f_t^\dagger R_t g_t\,dt.
$$
A theory canonical systems was developed by de Branges \cite{de1968hilbert} in conjunction with generalizing the concept of Fourier transform.

The Hilbert-P\'olya conjecture seeks to prove the Riemann hypothesis by finding a self-adjoint operator whose
eigenvalues are the zeros $Z$ of $\zeta(1/2+iz)$ for the Riemann zeta function $\zeta$. A famous attempt at proving the Riemann hypothesis was made by de Branges, using Dirac operators corresponding to canonical systems.

On the other hand, the Montgomery conjecture \cite{mont} claims that as $t\to infty$,
the random set $(Z-Ut)\log t$, where $U$ is a uniform random variable on $[0,1]$, converges to the Sine$_2$ process, defined as the limit of eigenvalue process of the GUE in the bulk.

A natural question is whether there exists an operator  (coming from canonical system) whose eigenvalues are give the Sine$_2$ process. The first theorem from \cite{VV3} answers this in the affirmative, for all $\beta$. The operator we describe here is conjugate to a canonical Dirac operator via a Cayley transform, see \cite{VV3}, but the present form is more convenient for analysis.

Consider the hyperbolic Brownian motion in the Poincar\'e disk satisfying the SDE
\begin{equation}\label{e:hbm}
d\mathcal B = \frac{1}{\sqrt{\beta(1-t})}(1-\mathcal B)dZ
\end{equation}
where $Z$ is a complex Brownian motion with independent standard real and imaginary parts, and the time scaling corresponds to logarithmic time. Let
\begin{equation}\label{e:Xdef}
X_t =\frac{1}{\sqrt{1-|\mathcal B(t)|^2}}
\mat{1}
{\mathcal B(t)}
{\overline{\mathcal B}(t)}
{1},
\qquad J=\mat{-i}{0}{0}{i}.
\end{equation}
Define the Brownian carousel operator as
\begin{equation}\label{e:carop}\mathcal C_\beta = J \,X_t^2\, \partial_t \qquad \mbox{on } [0,1).
\end{equation}
with boundary conditions $f(0)\parallel (1,1)^\dagger$ and
$f(1)\parallel (\mathcal B(\infty),1)^\dagger$ (since $\mathcal B$ converges to a point on the unit circle).
We will see that $2\mathcal C_\beta$ has a discrete set of eigenvalues with a translation-invariant distribution. It is called the $\Sineb$ process.

Then we have
\begin{theorem}[\cite{VV3}]\label{t:bulk}
Fix $\nu \in(-2,2)$.
There exists unitary matrices so that for the $\beta$-Hermite tridiagonal matrices $T_n$
$$
\sqrt{1-\nu^2}\;O_n(T_n-\nu I)O_n^{-1}\rightarrow \mathcal C_\beta
$$
where $T_n$ acts on the $\mathbb C^n$ as a subspace of complex 2-vector-valued functions on $[0,1)$. The convergence is in the norm-resolvent sense; in particular eigenvalues converge and eigenvectors converge in norm.
\end{theorem}
A version of this theorem, for unitary matrices (and for the associated phase function instead  of the operator) was given Killip and Stoicu \cite{KS}. In \cite{BVBV} a phase function version is proved. The full operator convergence is shown in \cite{VV3}.

\bigskip\noindent{\bf The Brownian carousel as a geometric evolution.} Writing the eigenvalue equation for $\mathcal C_\beta$ as
$$
\partial_t g= -\lambda J^{X_t^{-1}}g , \qquad g(0)=(1,1)^\dagger.
$$
Shows that $\mathcal P  g_t=e^{i\gamma_t}$, a point on the unit circle, is rotated at speed $\lambda$  about the moving center
$\mathcal P \mathcal B(t)$. In particular, $\gamma$ satisfies
\begin{equation}
\partial_t \gamma=\lambda \frac{ |e^{i \gamma}-\mathcal{B}|^2}{1-|\mathcal{B}|^2}, \qquad \gamma(0)=0.\label{odegamma}
\end{equation}
Oscillation theory tells us that the
number of eigenvalues in the interval $[0,\lambda]$ equals the number of times $e^{i\gamma}$ visits the point $\mathcal B(\infty)$. This process is called the Brownian carousel, introduced in \cite{BVBV} before the discovery of the operator $\mathcal C_\beta$.

We will not describe the proof of Theorem \ref{t:bulk} here. Instead, we
will explain how this operator arises as a limit of lifts of (random) unitary matrices. Then we present some applications to approximating eigenvalue statistics. Finally, we will discuss a related model, 1-dimensional critical random Schr\"odinger operators.

\section{An operator and a path associated with unitary matrices}

The goal of this section is to parameterize the spectrum of a unitary matrix in a way that it will be apparent already for finite $n$ what the limiting operator will be. In fact, we construct a Dirac operator whose spectrum is the lifting of that of $U$. Moreover, the operator depends on a piecewise constant path in the hyperbolic plane. If this path has a limit as $n\to\infty$ (and some tightness conditions are satisfied) then so will the associated Dirac operator.

As it turns out, in the circular beta case the parameter path is just a random walk in the hyperbolic plane! Hence the limit will be the operator parameterized by hyperbolic Brownian motion.

The construction is based on the Szeg\H o recursion, which we will briefly review here.

Let $U$ be a unitary matrix of dimension $n$, and assume that for some unit vector $e$, the
vectors $e, Ue, \ldots U^{n-1}e$ form a basis. There is a unique way to apply Gram-Schmidt
to orthonormalize this basis so that we get
$$
\Phi_0(U)e, \ldots, \Phi_{n-1}(U)e
$$
where $\Phi_k$ is a monic degree $k$ polynomial. Define $\Phi_n$ to be monic of degree $n$ so that
$\Phi_n(U)e=0$; this implies that $\Phi_n(z)=\det(z-U)$ the characteristic polynomial of $U$.  Writing
$$\Phi_k(z)=z^k+a_{k-1}z^{k-1} +\ldots +a_0$$ we define
$$\Phi^*_k(z)=\bar a_0z^k +\bar a_1z^{k-1} +\ldots +\bar a_k=z^{k}\overline{\Phi_k(1/\bar z}).$$
Now note that
$$\langle \Phi_k^*(U)e, U^j e \rangle=\sum_{i=0}^k
\bar a_i \langle U^{k-i}e,U^j e\rangle = \sum_{i=0}^k
\bar a_i \langle U^{k-j}e,U^{i} e\rangle = \overline{\langle\Phi_{k}(U)e,U^{k-j}e}\rangle.
$$
By construction, $u=\Phi_k(U)e$ is perpendicular to $e,\ldots, U^{k-1}e$, it follows
that $\Phi^*_k(U)e$ is perpendicular to $Ue,\ldots, U^{k}e$. However,
so is $v=\Phi_{k+1}(U)e-U\Phi_{k}(U)e$ (as each term is, by construction).
Now $u,v$ are in the span of $e,\ldots, U^ke$, so they must
be collinear. Following tradition we choose $\alpha_k$, the so-called Verblunski coefficients,  so that
\begin{equation}
  \label{rec}
\Phi_{k+1}-z\Phi_{k} = -\bar\alpha_k \Phi^*_k,
\end{equation}
namely
$$
-\bar \alpha_k = \frac{\langle u,v \rangle}{\langle u,u\rangle} =
\frac{\langle \Phi_k^*(U)e,-U\Phi_k(U)e\rangle }{\|\Phi^*_k(U)e\|^2}.
$$
Since $\Phi_k^*(U)e$ and $U\Phi_k(U)e$ have the same length, we see that
$|\alpha_k|\le 1$.
We then get the celebrated
Szeg\H o recursion
\begin{align*}
\binom{\Phi_{k+1}(z)}{\Phi^*_{k+1}(z)}=A_k Z\binom{\Phi_{k}(z)}{\Phi^*_{k}(z)}, \qquad \Phi^*_0(z)=\Phi_0(z)=1,
\end{align*}
with the matrices
\begin{align*}
\qquad A_k=\mat{1}{-\bar \alpha_k }{-\alpha_k}{1}
, \qquad Z=\mat{z}{0}{0}{1}.
\end{align*}
Note that $z$ is an eigenvalue if and only if $\Phi_n(z)=0$, equivalently by \eqref{rec} we have
\begin{equation}\label{e:evcrit}
Z \binom{\Phi_{n-1}(z)}{\Phi^*_{n-1}(z)}=ZA_{n-2}Z\cdots Z A_0 Z \binom{1}{1} \parallel
\binom{\bar \alpha_{n-1}}{1}.
\end{equation}

Using the Verblunski coefficients, we can define a new set of parameters
\begin{equation}\label{e:bcoord}
b_k = \mathcal P A_{0}^{-1} \ldots A_{k-1}^{-1} \binom{0}{1}, \qquad 0\le k <n-1
\end{equation}
where $\mathcal P \binom{x}{y}=x/y$, and
$$
b_* = \mathcal P A_{0}^{-1} \ldots A_{n-2}^{-1} \binom{\bar \alpha_{n-1}}{1}.
$$
Then $b_0=0$ and the parameters $(b_1,\ldots, b_{n-1},b_*)$ encode the same information as the $\alpha_i$. This is exactly the information contained in the spectral measure $\sum_{j=1}^n w_j \delta_{e^{i\lambda_j}}$.
\begin{theorem}
Consider the measure $\sum_{j=1}^n w_j \delta_{e^{i\lambda_j/n}}$  supported on $n$ points on the unit circle, and consider the $b$-coordinates \eqref{e:bcoord}. For $t\in[0,1]$ let $b(t)=b_{\lfloor tn\rfloor}$, and let
$$X_t =\frac{1}{\sqrt{1-|b(t)|^2}}\mat{1}{b(t)}{\bar b(t)}{1}, \qquad J=\mat{-i}{0}{0}{i}.
$$
 Then the operator
\begin{equation}\label{e:uoperator}
JX^2_t\partial_t
\end{equation}
acting on functions $f:[0,1]\rightarrow \mathbb C^2$ with the
boundary conditions $f_1(0)=f_2(0)$ and $f_1(1)=f_2(1)b_*$ has discrete spectrum and the eigenvalues
are  $\lambda_i/2+ \pi n \mathbb Z$.
\end{theorem}
\begin{proof}
We skip the standard proof of self-adjointness, see \cite{VV3}.
Instead of the Szeg\H o recursion, we can follow the evolution of
$$
\Gamma_k=Z^{A_{k-2}\ldots A_0} \cdots Z^{A_0}Z\binom{1}{1},
$$
so that
$$\Gamma_0=\binom{1}{1},\quad \Gamma_1=Z\binom{1}{1},\quad \Gamma_2=Z^{A_0}Z\binom{1}{1},\quad\ldots$$
which, geometrically is a repeated rotation of the vector around
a moving center given by $b_k$, and
$$\Gamma_{k+1}=Z^{A_{k-1}\ldots A_0} \Gamma_k  =Z^{X_{k/n}^{-1}}\Gamma_k
$$
Since $J$ is an infinitesimal rotation element around $0$, with $z=e^{i\lambda/n}$ the solution $\Gamma(t)$ of the ODE
$$
\partial_t \Gamma(t) = -\frac {\lambda}{2} J^{X_t^{-1}} \Gamma(t),   \qquad \Gamma(0) =\binom{1}{1}
$$
satisfies $\Gamma(k/n)=e^{-ik/2n}\Gamma_k$ for $k=0,\ldots, n$. But since $X_tJX_t^*=J$, $X_t=X_t^*$ and $J^2=-I$, this ODE is just the eigenvalue equation at $\lambda/2$ of $JX_t^2\partial t$. Note also that $\Gamma(1)$ is parallel to the middle term of \eqref{e:evcrit}, so the boundary condition is also correct. \end{proof}

\section{The path parameter for circular $\beta$}

We now look at the circular $\beta$ ensembles. Their joint eigenvalue density is proportional to  Vandermonde to the power $\beta$. What we need is that for this eigenvalue distribution we can take the  $\alpha_k$ to be rotationally symmetric, independent with
$$
|\alpha_k^2|\sim \operatorname{Beta}\big[1,(n-k-1)\beta/2\big]
$$
with $\alpha_{n-1}$ uniform on the circle, as shown by Killip and Nenciu \cite{KillipNenciu}.
The evolution of $b_k$ is
$$
b_{k+1}=A_k^{A_{k-1}\cdots A_0}.b_k
$$
where the $A_k$ are now to be understood as linear fractional transformations, or, equivalently, hyperbolic automorphisms in the Poincar\'e model.

Note that $A_k$ moves the origin to a rotationally invariant random location, and so $A_k^{A_{k-1}\cdots A_0}$ moves $b_k$ to a rotationally invariant random location around $b_k$. In particular, $b_k$ is just a random walk in the hyperbolic plane that can be described alternatively as follows. Let $b_0=0$. Given $b_k$, pick a point uniformly on the hyperbolic circle around $b_k$ whose radius equals the hyperbolic distance $d_k$ of $0$ and a random variable with the same distribution as $|\alpha_k|$.

Given a hyperbolic Brownian motion path $B$, this method suggest an efficient coupling. First pick $d_1,\ldots, d_{n-1}$, let $b_0=0$, $t_0=0$, and given $b_k,t_k$ let $t_{k+1}$ be the first time that
$\operatorname{dist}(B_t,b_k)=d_{k+1}$. Let $b_{k+1}=B(t_{k+1})$.

Given this coupling, it is now straightforward to show that the path
$b_n(t) \to \mathcal B(t)$ a.s. uniformly on compacts, for $\mathcal B$ defined in \eqref{e:hbm}.
With an additional tightness argument, we get
\begin{theorem}[\cite{VV3}]
The operators $\mathcal C_{\beta,n}$ defined by \eqref{e:uoperator} with paths $b_n$ coupled as above, converge
in the norm-resolvent sense to the limit $\mathcal C_\beta$ of \eqref{e:carop}. In particular, the circular $\beta$ eigenvalue process converges to the eigenvalues of $C_\beta$.
\end{theorem}

For bulk results in the Laguerre case, see \cite{jacquot2011bulk}.

\section{The Brownian carousel}

The Brownian carousel description gives a simple way to
analyze the limiting point process. The hyperbolic angle of
the rotating boundary point as measured from $b(t)$ follows
the \textbf{Brownian carousel SDE}. Indeed, define
$\alpha_\lambda(t)$ to be the continuous function with $\alpha_\lambda(0)=0$ so that with $X$ as in \eqref{e:Xdef} (recall $\mathcal P (x,y)^\dagger = x/y$)
$$
e^{i\alpha_\lambda(t)}=\mathcal P X^{-1} g_\lambda(t)
$$
for the solution $g_\lambda$ of the ODE $2C_\beta g_\lambda = \lambda g_\lambda$ started at $(1,1)^\dagger$. (A factor 2 here for backward compatibility). While $\mathcal P g$ evolves monotonously on the circle, the evolution of $\alpha$ satisfies a coupled one-parameter
family of stochastic differential equations. We apply a logarithmic time change for simplicity to get, with $f(t)= \frac{\beta}{4}\exp(-\beta t/4)$ the SDE
\begin{equation}\label{e_ssefenetudja}
d\alpha_\lambda=  \lambda  f\,dt + \Re ((e^{-i\alpha_\lambda}-1)dZ), \qquad
\alpha_\lambda(0)=0,
\end{equation}
driven by a two-dimensional standard Brownian motion. For a
single $\lambda$, this reduces to the one-dimensional
stochastic differential equation
\begin{equation}\label{e sse3a}
d\alpha_\lambda=   \lambda f\,dt + 2 \sin(\alpha_\lambda/2)dW,\qquad
\alpha_\lambda(0)=0,
\end{equation}
which converges as $t\to\infty$ to an integer multiple
$\alpha_\lambda(\infty)$ of $2\pi$. A direct consequence of
oscillation theory for $\mathcal C_\beta$ is the following.
\begin{proposition}\label{sseprop}
The number of points $N(\lambda)$ of the point process $\Sineb$ in
$[0,\lambda]$ has the same distribution as
$\alpha_\lambda(\infty)/(2\pi)$.
\end{proposition}
%

\section{Gap probabilities}

In the 1950s Wigner examined the asymptotic probability of having no eigenvalue in a fixed interval of size $\lambda$ for $n\to \infty$ while the spectrum is  rescaled to have an average eigenvalue spacing $2\pi$.
Wigner's prediction for this probability was
$$
p_\lambda= \exp\left(- (c +o(1))\lambda^2\right).
$$
where this is a $\lambda\to \infty$ behavior.
%
This
rate of decay is in sharp contrast with the exponential
tail for gaps between Poisson points; it is one
manifestation of the more organized nature of the random
eigenvalues. Wigner's estimate of the constant $c$,
$1/(16\pi)$, later turned out to be inaccurate. \cite{Dy62}
improved this estimate to
\begin{equation}\label{mainform}
 p_\lambda
=(\kappa_\beta+o(1))\lambda^{\gamma_\beta}\,\exp\left(-
\frac{\beta}{64}\lambda^2+\left(\frac{\beta}{8}-\frac14\right)\lambda\right)
\end{equation}
which applies to the $\Sineb$ process.

Dyson's computation of the exponent $\gamma_\beta$, namely
$\frac{1}{4}(\frac{\beta}{2}+\frac{2}{\beta}+ 6)$, was
shown to be
slightly incorrect. Indeed,  \cite{dCM73} gave more substantiated predictions that $\gamma_\beta$ is equal to $-1/8, -1/4$ and $-1/8$ for values $\beta=1, 2$ and 4, respectively. 
 Mathematically precise proofs for the $\beta=1,2$ and 4 cases were later given by several authors: \cite{Wi96}, \cite{DIZ96}. Moreover, the value of $\kappa_\beta$ and higher order asymptotics were also established for these specific cases by \cite{Kr04}, \cite{Ehr06}, \cite{DIKZ07}.

In \cite{BVBV2} we give a mathematically rigorous version of Dyson's  prediction for general $\beta$ with a corrected exponent $\gamma_\beta$ using the Brownian carousel SDE.
\begin{theorem}\label{mainthm}
The formula \eqref{mainform} holds
with a positive $\kappa_\beta$ and
$$
\gamma_\beta=\frac{1}{4}\left(\frac{\beta}{2}+\frac{2}{\beta}-
3\right).
$$
\end{theorem}
We include a proof of a theorem from \cite{BVBV} that works for more general driving functions $f$ (the equation \eqref{e_ssefenetudja}) but gives a weaker result in this case, namely the main order term in the upper bound.
\begin{theorem}\label{t_ggap}
Let $f:{\mathbb R}^+\to {\mathbb R}^+$  satisfy $f(t)\le c/(1+t^2)$ for all $t$ and
$\int_0^\infty |df|<\infty$. Let $k\ge 0$. As $\lambda \to
\infty$, for the point process given by the Brownian carousel with
parameter $f$ we have
\begin{eqnarray} \label{e_ggap}
\pr(\mbox{\# of points in }[0,\lambda]\le k) =
\exp\big(-\lambda^2( \|f\|_2^2/8+o(1))\big).
\end{eqnarray}
\end{theorem}

\begin{lemma}\label{gtail} Let $Y$ be an adapted stochastic process with $|Y_t|<m$, and let
$X$ satisfy the SDE $dX=YdB$ where $B_t$ is a Brownian motion.
Then for each $a,t>0$ we have
$$
\pr(X(t)-X(0)\ge a) \le \exp\left(- a^2/(2 tm^2) \right).
$$
\end{lemma}
\begin{proof} We may assume $X(0)=0$.
Then $X_t=B_\tau$ where $\tau$ is the random time change
$\tau=\int_0^t Y^2(s)ds$. Since $\tau<m^2 t$ the inequality
now follows from
\begin{equation*}
\pr(B_r>a)\le \exp\left(- a^2/(2 r) \right).\qedhere
\end{equation*}
%
%
%
%
\end{proof}

\begin{proof}[Proof of Theorem \ref{t_ggap}] The event in
\eqref{e_ggap} is given in terms of the Brownian carousel SDE
as $\lim_{t\to\infty} \alpha_{\lambda}(t)\le 2k\pi$.

Since $\alpha(t)$ never returns below a multiple of $2\pi$ that it has passed, it is enough to give an upper bound on the
probability that $\alpha$ stays less than $x=2(k+1)\pi$.
 For $0<s<t$ we have
\begin{eqnarray*}
\pr(\alpha(t)<x\,|\,\mathcal F_{s})  &=& \pr\left(-\int_s^t
2\sin(\alpha/2) dB \,> \,\lambda \int_s^t f dt - x
+\alpha(s)\,\Big|\,\mathcal F_{s}\right).
\end{eqnarray*}
We may drop the $\alpha(s)$ from the right hand side and use Lemma
\ref{gtail} with $Y=-2\sin(\alpha/2)$, $m=2$, $a=\lambda (\int_s^t f dt- x /\lambda)$ to get the upper bound
$$
\pr(\alpha(t)<x\,|\,\mathcal F_{s})\le \exp(-\lambda^2r(s,t)),\qquad
 r(s,t)=\frac{(\int_s^t f dt- x /\lambda)^2}{8(t-s)}.
$$
Then, by just requiring $\alpha(t)<x$ for times $\eps,
2\eps,\ldots \in[0,K]$ the probability that $\alpha$ stays less
than $x=2(k+1)\pi$ is bounded above by
$$
 \ev
\prod_{k=0}^{K/\eps} \pr(\alpha((k+1)\eps)<x\,\big|\,\mathcal
F_{k\eps})
 \le \exp\Big\{-\lambda^2 \sum_{k=0}^{K/\eps} r(\eps k, \eps
k+\eps) \Big\}.
$$
A choice of $\eps$ so that  $ x /\lambda=o(\eps)$ as
$\lambda\to\infty$ yields
 the asymptotic Riemann sum 
$$
\sum_{k=0}^{K/\eps} r(\eps k, \eps k+\eps) =\frac{1}{8}
\int_0^K f^2(t)dt + o(1).
$$
Letting $K\to \infty$ provides the desired upper bound.
\end{proof}
Next, we show a central limit theorem for the number of eigenvalues of $\mathcal C_\beta$ from \cite{KVV}.
\begin{theorem}[CLT for $\Sineb$]
\label{cltsineb}
As $\lambda\to \infty$ we have
\[
\frac1{\sqrt{\log \lambda}} \of{\textup{Sine}_\beta[0,\lambda]-\frac{\lambda}{2\pi}} \Rightarrow \cN(0,\frac{2}{\beta \pi^2}).
\]
\end{theorem}
An $n\to\infty$ version of this theorem for finite matrices
from circular and Jacobi $\beta$ ensembles was shown by Killip \cite{killip2008gaussian}.

\begin{proof}
We will consider the Brownian carousel SDE
\begin{equation}\label{sineb1}
d \alpha^{\lambda}=\lambda \frac{\beta}{4}
e^{-\frac{\beta}{4}t} dt+ 2 \sin(\alpha^{\lambda}/2) d{B},
\quad \alpha^\lambda(0)=0\quad t\in[0,\infty).
\end{equation}
First note that $\tilde \alpha(t)=\alpha^{\lambda}(T+t)$
with $T=\frac{4}{\beta} \log(\beta \lambda/4)$ satisfies
the same SDE with $\lambda=1$. Therefore
$$\frac{\alpha^{\lambda}(\infty)-\alpha^\lambda(T)}{\sqrt{\log(\lambda)}}\to 0$$
in probability. So it suffices to find the the weak limit
of $$\frac{\alpha^{\lambda}(T)-\lambda}{2\pi \sqrt{\log
\lambda}}.$$

We have  $$\alpha(T)-\lambda=-\frac{4}{\beta}+\int_0^T
2\sin(\alpha^{\lambda}/2) d{B}$$
which means
$$\alpha(T)-\lambda+\frac{4}{\beta}\eqd \hat {B}\of{\int_0^T
4\sin(\alpha^{\lambda}/2)^2 dt}$$ for a certain standard
Brownian motion $\hat {B}$. In order to prove the required
limit in distribution
 we only need to show that $\frac{4}{\log \lambda}\int_0^T  \sin(\alpha^\lambda/2)^2 dt\to \frac{8}{\beta}$ in probability. We have
\[
\frac{4}{\log \lambda}\int_0^T  \sin(\alpha^\lambda/2)^2 dt=\frac{8 \log\left[\beta  \lambda /4\right]}{\beta \log \lambda}+\frac{2}{\beta \log \lambda} \int_0^T \cos(\alpha^\lambda) dt.
\]
The first term converges to $8/\beta$.  To bound the second term we compute
\begin{eqnarray*}
\frac{4}{i \beta \lambda \log \lambda} d\of{e^{i \alpha^\lambda+\beta t/4}}&=&\frac{e^{i \alpha^\lambda}}{\log \lambda} dt+\frac{8}{\beta \lambda \log \lambda} e^{i \alpha^\lambda+\beta t/4} \sin(\alpha^\lambda/2) d{B}\\
&+&\frac{8i}{\beta \lambda \log \lambda} e^{i \alpha^\lambda+\beta t/4} \sin(\alpha^\lambda/2) ^2 dt\\&+&\frac{1}{i \lambda \log \lambda} e^{i \alpha^\lambda+\beta t/4}dt.
\end{eqnarray*}
The integral of the left hand side is $
\frac{4}{i \beta \lambda \log \lambda}\left[4 e^{i \alpha^\lambda(T)} \lambda /\beta- 1 \right]=O((\log \lambda)^{-1})
$. The integrals of the last two terms in the right hand side are of the order of $(\lambda \log \lambda)^{-1} \int_0^T e^{\beta t/4} dt=O((\log \lambda)^{-1})$. Finally, the integral of the second term on the right has an $L^2$ norm which is bounded by $C (\log \lambda)^{-1}$. This means the integral of the first term on the right, $(\log \lambda)^{-1} \int_0^T e^{i \alpha^\lambda} dt$ converges to 0 in probability from which the statement of the theorem follows.
\end{proof}

\section{Random Schr\"odinger limits}

The methods developed for tridiagonal matrices also work
for critical 1-dimensional random Schr\"odinger operators. It is interesting to compare the behavior of level statistics.

Consider the matrix
\begin{equation}\label{shrod1dmatrix}
H_n=\left( \begin{array}{cccccc}
v_1 & 1 &  &  &  & \\
1 & v_{2} & 1 &  & & \\
  & 1  &\ddots &\ddots & &\\
& & \ddots & \ddots &1 & \\
& & & 1 & v_{n-1} & 1 \\
& & & & 1 & v_{n} \\
\end{array} \right)
\end{equation}
where $v_k=\sigma \omega_k/\sqrt{n}$, and  $\omega_k$ are independent random variables with mean
$0$, variance $1$ and bounded third absolute moment.

To cut a long story short, one can take a limit of this operator around the global position $E$ just as the $\beta$-Hermite models in Theorem \ref{t:bulk}. The resulting
operator $\mathcal S_\tau$ is an analogue of $\mathcal C_\beta$, except it is driven
by time-homogeneous hyperbolic Brownian motion on an interval of length $\tau=\sigma^2/(1-E^2/4)$ is the only parameter left in the process. In \cite{KVV} we show that the large gap probabilities have a similar behaviour (exponentially decaying in the square of the gap) to the $\Sineb$ process (see also \cite{holcomb2013large} for more detailed large deviation results).

The CLT and the level repulsion are different, indicating much higher ordering. We include the geometric proof of the repulsion here, using the Brownian carousel description of Section \ref{s:carousel}. Let $\sch[I]$ denote the number of eigenvalues of the operator $\tau \mathcal S_\tau$ in the interval $I$.

\begin{theorem}[Eigenvalue repulsion, \cite{KVV}]\label{repulsion2} For $\eps>0$
we have
\begin{equation}\label{bound2eig2}
\Prob\set{\sch[0,\eps]\geq 2}\leq
4\exp\left(-\frac{(\log(2\pi/\eps)-\tau-1)^2}{\tau}\right).
\end{equation}
whenever the squared expression is nonnegative.
\end{theorem}

\begin{proof}
If there are at least  two points in $[0,\eps]$ then the Brownian carousel had to take at least one full turn. Thus
\[
\Prob\set{\sch[0,\eps]\geq 2}\leq \Prob \set{\gamma^{
\eps/\tau}\of{\tau}\geq 2\pi}.
\]
where $\gamma$ is the solution of \eqref{odegamma}. From
(\ref{odegamma}) we get
\[
\gamma^{\eps/\tau}(\tau)\le  \eps \max_{0\le t \le \tau}
(1-|{B}_t|^2)^{-1}=\eps (1-\max_{0\le t \le \tau}|{B}_t|^2)^{-1}
\]
which means that
\begin{equation}
\gamma^{\eps/\tau}\of{\tau}\geq 2\pi \quad\Rightarrow\quad
1-\frac{\eps}{2 \pi}\le \max_{0\le t\le \tau}
|{B}_t|^2.\label{aux1}
\end{equation}
In the Poincar\'e disk model the hyperbolic distance between the origin and a point $z$ in the unit disk is given by $q(z)=\log\of{\frac{1+|z|}{1-|z|}}$. Thus (\ref{aux1}) implies
\[
\max_{0\le t\le\tau} q({B}_t)\ge \log\of{2\pi/
\eps}.
\]
The probability that the hyperbolic Brownian motion leaves
a ball  with a large radius $r$ in a fixed time is
comparable to the probability that a one-dimensional
Brownian motion  leaves $[-r,r]$ in the same time. This
follows  by noting that It\^o's formula with
(\ref{e:hbm}) gives
\[
dq =\frac{dB}{\sqrt{2}}+ \frac{\coth(q)}{4} dt
\]
for the evolution of $q({B})$ with a standard Brownian
motion $B$. By increasing the drift from $\coth(q)/4$ to
$\infty {\bf 1}_{q\in [0,1]}+\coth(1)/4$ we see that  $q$
is stochastically dominated by $1+t\coth(1)/4 +
|B(t)|/\sqrt{2}$ where $B$ is standard Brownian motion and
$\coth(1)<4$. Thus
\begin{eqnarray*}
\Prob\of{\max_{0\le t\le\tau } q({B}_t)\ge
\log\of{2\pi/ \eps}}&\le&
\Prob\of{\max_{0\le t\le\tau} |B(t)|\ge \log\of{2\pi/\eps}-1-\tau}\\
&\le&
4\exp\left(-\frac{(\log(2\pi/\eps)-\tau-1)^2}{\tau}\right)
\end{eqnarray*}
which proves the theorem.
\end{proof}
We note that continuum random Schr\"odinger models can also have such limits, see \cite{kona} and \cite{nakano}.

Most of this review was about eigenvalues. To conclude, we include a remarkable fact about the shape of localized eigenvectors of 1-dimensional random Schr\"odinger operators, \cite{Ben}.

\begin{theorem}\label{global_eigenvector}
Pick $\lambda$ uniformly from the eigenvalues of $H_n$ and let $\psi^\lambda$ be the corresponding normalized eigenvector. Let $B$ be a two sided Brownian motion started from $0$, and let
$$M(t) = \exp(B(t)-|t/2|).$$
Then, letting $\tau_E = \sigma^2/(1-E^2/4)/4$, as $n\to\infty$ we have the convergence in joint distribution
$$
\Big(\; \lambda, \;\psi^{\lambda} \lfloor t / n \rfloor ^2 dt^*  \Big)
\;\Longrightarrow\; \Big( \;E, \;M(\tau_E(t-U))dt^*\;\Big)
$$
where $E$ has arcsin distribution on $[-2,2]$, $U$ is uniform on $[0,1]$, and $E,U,M$ are independent. Here $dt^*$ signifies that the measures are both normalized to have total mass 1.
\end{theorem}

\section{Further open problems}

These are in addition to the problems and questions presented in the body of the article.

\begin{question}[Decimation]
In \cite{decimation} it was shown that deleting all but every $k$th eigenvalue of many finite $\beta=2/k$ ensembles gives the corresponding $\beta=2k$ ensemble. Can the limiting operators (bulk or edge) be coupled explicitly in this way?
\end{question}

\begin{question}[Random and deterministic orthogonal polynomials]
Is there a relation between the $\beta=2$ random orthogonal polynomials (see section \ref{s:randommeasure}) and the deterministic ones? How about the limiting operators?
\end{question}

\begin{question}[Dynamics]
Are there operator limits of matrix-valued (say Hermitian) Brownian motion?
\end{question}

\begin{question}[Painlev\'e in the bulk]\label{q:bulkpain}
Can one deduce the gap Painlev\'e equation from the PDE's corresponding to the generator of the Brownian carousel SDE?
\end{question}

\begin{question}[Loop equations]
Can one derive analogues of the loop equations directly from limiting operators?
\end{question}

\noindent {\bf Acknowledgments.} The author is grateful to the R\'enyi Institute, Budapest
for its hospitality during the writing of this paper.


\def\cprime{$'$} \def\cprime{$'$}

\end{document}